\documentclass[12pt]{amsart}

\usepackage[english]{babel}

\usepackage{times,amsfonts,amsmath,amssymb,dsfont} 
\usepackage{pdfsync}
\usepackage{graphicx} 
\usepackage{mathabx}

\usepackage{cases}

\usepackage{soul}
\setstcolor{red}

\usepackage{xcolor}
\usepackage[normalem]{ulem}
\newcommand{\stkout}[1]{\ifmmode\text{\sout{\ensuremath{#1}}}\else\sout{#1}\fi}

\definecolor{gr}{rgb}  {0.,  0.69,  0.23 }
\definecolor{bl}{rgb}  {0.,  0.5,  1. }
\definecolor{mg}{rgb}  {0.85, 0.,  0.85}
\definecolor{or}{rgb}  {0.9, 0.5,  0.}

\definecolor{webred}{rgb}{0.75,0,0}
\definecolor{webgreen}{rgb}{0,0.75,0}
\usepackage[citecolor=webgreen,colorlinks=true,linkcolor=webred]{hyperref}

\newtheorem{theorem}{Theorem}
\newtheorem{proposition}[theorem]{Proposition}
\newtheorem{lemma}[theorem]{Lemma}

\theoremstyle{definition}

\theoremstyle{remark}
\newtheorem{remark}[theorem]{Remark}

\newcommand{\Bk}{\color{black}}

\newcommand{\N}{\mathbb{N}}
\newcommand{\R}{\mathbb{R}}
\newcommand{\C}{\mathbb{C}}

\newcommand{\one}{\mathds{1}}

\newcommand{\cA}{\mathcal{A}}
\newcommand{\cB}{\mathcal{B}}

\newcommand{\cD}{\mathcal{D}}
\newcommand{\cE}{\mathcal{E}}
\newcommand{\cF}{\mathcal{F}}
\newcommand{\cL}{\mathcal{L}}

\newcommand{\cO}{\mathcal{O}}

\newcommand{\gp}{\mathfrak{p}}

\newcommand{\bel}{\begin{equation} \label}
\newcommand{\ee}{\end{equation}}

\newcommand{\ba}{\begin{array}}
\newcommand{\ea}{\end{array}}

\newcommand{\pf}{\partial_\varphi}

\newcommand\ess{\operatorname{ess}}
\newcommand\disc{\operatorname{disc}}

\newcommand\ie{\emph{i.e.,\ } }

\begin{document}\setul{2.5ex}{.25ex}
\title[]{
Eigenvalue and Resonance Asymptotics in perturbed periodically twisted tubes: Twisting versus Bending}
\author[V.\ Bruneau]{Vincent Bruneau}
\address{Universit\'e de Bordeaux, IMB, UMR 5251, 33405 TALENCE cedex, France}
\email{Vincent.Bruneau@u-bordeaux.fr}
\author[P.\ Miranda]{Pablo Miranda}\address{Departamento de Matem\'atica y Ciencia de la Computaci\'on, Universidad de Santiago de Chile, Las Sophoras 173. Santiago, Chile.}\email{pablo.miranda.r@usach.cl}
\author[D.\ Parra]{Daniel Parra}
\address{Graduate School of Mathematical Sciences, University of Tokyo, 3-8-1 Komaba, Meguroku, Tokyo,153-8914, Japan }
\email{dparra@ms.u-tokyo.ac.jp}
\author[N.\ Popoff]{Nicolas Popoff}
\address{Universit\'e de Bordeaux, IMB, UMR 5251, 33405 TALENCE cedex, France}
\email{Nicolas.Popoff@u-bordeaux.fr}

\maketitle
\begin{abstract}
We consider the Dirichlet Laplacian in a three-dimensional waveguide that is a small deformation of a periodically twisted tube. The deformation is given  by a bending and an additional twisting of the tube, both   parametrized by a coupling constant $\delta$. We expand the resolvent of the perturbed operator near the bottom of its essential spectrum and we show the existence of exactly one resonance, in the asymptotic regime of $\delta$ small. We are able to perform the asymptotic expansion of the resonance in $\delta$, which in particular  permits \Bk us to give a quantitative geometric criterion for the existence of a discrete eigenvalue below the essential spectrum. In the particular case of  perturbations of  straight tubes, \Bk we are able to show the existence of resonances not only near the bottom of the essential spectrum but near each threshold in the spectrum. We also obtain the asymptotic behavior of the resonances in this situation, which is generically different from the first case.
\end{abstract}

\section{Introduction}
In this article, we consider the Dirichlet Laplacian in a three-dimensional waveguide $\Omega\subset \R^3$.  The waveguide is given by an infinite tube constructed as follows: let $\gamma:\R \to \R^3$ be a regular curve with curvature and torsion $\kappa$ and $\tau$, respectively. We assume that this curve is asymptotically straight, in the sense that $\kappa$ and $\tau$ go to 0 at infinity. Let $ \omega\subset \R^2$ be a fixed domain and construct  the following tube:  along the curve  $\gamma$  put the  cross section $\omega$ in such a way that it is orthogonal to $\dot{\gamma}(s)$ and rotated in an angle $\theta(s)$  around  the same vector. We assume that  $\theta$ is a smooth perturbation of a constant rotation, in the sense that $\dot{\theta}(s)=\beta+\varepsilon(s)$, where $\beta\in\R$, and $\varepsilon$ is a decaying function. The resulting domain is called $\Omega$ (we refer to the next section for a  precise definition), and asymptotically, it is a perturbation of a periodically twisted tube $\Omega_0$ (corresponding to $\kappa=\tau=\varepsilon=0$). We will study the Dirichlet Laplacian in $\Omega$ as a perturbation of the Dirichlet Laplacian in $\Omega_0$.  Using a  change of variables  we  study \Bk the corresponding unitary equivalent 
differential operators defined both  in the straight tube $\R\times\omega$. We denote by $H_0$ the operator corresponding to $\Omega_0$, and by $H$ the operator corresponding to $\Omega$.

An advantage of this procedure is that the operator $H_0$ can be fibered through partial Fourier transform with respect to the longitudinal variable. The spectrum of $H_0$ is absolutely continuous, and the ground state energy corresponds to a unique minimum $\mathcal{E}_1$ for the first band function (notice that very  little \Bk information is available on the band functions when $\beta\neq0$, see \cite{BriKovRaiSoc09}). The geometric perturbations considered here do not modify the essential spectrum i.e., the essential spectrum of the Laplacian in $\Omega$ is the same as the essential spectrum of the Laplacian in $\Omega_0$ (or $\sigma_{\rm{ess}}(H_0)=\sigma_{\rm{ess}}(H)$). However, discrete eigenvalues can appear. To know whether the geometric perturbation of a waveguide creates eigenvalues \Bk below the essential spectrum is a widely studied problem, corresponding to the appearance of trapped modes for a quantum particle moving in the tube. Let us cite \cite{ExKov15} for an overview of this problem.

For the case $\beta=0$, it is well known that the effective twisting, given by $\dot{\theta}-\tau$, and the bending, given by $\kappa$, compete together in this question: if $\kappa\neq0$ and the twisting effect is zero (i.e., $\tau-\dot{\theta}=0$), eigenvalues below the essential spectrum appear (\cite{DuEx95,CheDuFrei05}). This is also the case if the bending is strong enough. In contrast, the spectrum is unchanged in the case of pure twisting, which is related to the existence of a Hardy inequality provided by the twisting (\cite{EkKovKre08}). Furthermore, if the bending is small compared to the twisting, there is no discrete spectrum (for a comparison of both effects quantitatively in an asymptotic regime see \cite{Gru05}).

The case $\beta\neq0$ has also been studied before. Assume for instance $\beta>0$, and $\kappa=0$, then it is known that a non zero slowing twist (\ie $\int\varepsilon<0$) will create discrete spectrum (\cite{ExKov05, BriKovRaiSoc09}), but a small enhancement of the twisting ($\varepsilon\geq0$) should  not change the spectrum (see \cite{BrHaKr15} for a partial result). However, to the best of our knowledge, there is no study of the spectrum when both $\kappa\neq0$ and $\beta\neq0$.

This analysis raises the following  issues\Bk:
 
\begin{itemize}
\item Is it possible to have a quantitative  criterion \Bk to ensure the presence of a discrete eigenvalue below the essential spectrum?
\item When there are no eigenvalues, does the perturbation create resonances near the bottom of the spectrum? 
\item What happens near the upper thresholds in the spectrum?
\end{itemize}
Our approach to these questions is perturbative: for a parameter $\delta>0$, we will consider the same problem as above but replacing $\kappa$, $\tau$ and $\varepsilon$ by  $\delta \kappa$, $\delta\tau$ and $\delta\varepsilon$. We denote by $H_{\delta}$ the resulting operator. Then we will first study what happens near $\mathcal{E}_1$ when $\delta\to0$. This approach was already used in \cite{Gru04, Gru05}, but these works include only the case $\beta=0$. Moreover, they are concerned only with the existence of eigenvalues, and the study of the resonances is not treated. We have tackled these questions in the previous article \cite{BruMirPof18} when both $\beta$ and $\kappa$ are $0$ (i.e. when $\Omega_0=\R\times\omega$ and the deformation is only given by a twisting effect).

In this article, we continue this work with a more general treatment by considering the whole set of geometric perturbations described above. To do this we need to extend the resolvent $(H_\delta-z)^{-1}$, acting in weighted spaces, for $z$ near $\mathcal{E}_1$ on a two-sheeted Riemann surface. Then we show that for $\delta$ small enough there is exactly one pole of the resolvent near the bottom of the spectrum.  Our method also permits us to obtain the asymptotic expansion of this pole as $\delta\to0$, and therefore to deduce criteria for which it is a discrete eigenvalue or a resonance on the second sheet. When $\beta\neq0$, the distance from this pole to  $\mathcal{E}_1$ is in general of order $\delta$ in the variable $k = \sqrt{\mathcal{E}_1- z}$,  see Theorem \ref{T1} for the expression of the asymptotics. Moreover, when $\beta\neq0$, due to an interaction with the constant twisting, the presence of the bending does not necessarily contribute to the discrete spectrum, see the term in $\kappa$ in \eqref{12sep18} and Section \ref{comments}. This effect can be compared to the interaction of a bending with a mixed boundary condition, see \cite{BakEx18} for a recent result on this point, where the  existence 
of bending does not automatically create eigenvalues below the essential spectrum. 

On the other side, despite the fact that  various works on the existence of bound states for deformed tubes are available, the results are mainly concerned with their existence below the bottom of the  essential spectrum of $H_0$. The spectral effects that the geometric deformations have near the upper thresholds are not well studied (see however \cite{BruMirPof18}). Using our methods we can prove some results concerning this problem. In particular, in the case $\beta=0$ we know that the band functions of the operator $H_0$ are of the form $p\mapsto \cE_n+p^2$, for some $\cE_n\in \R, n\in\N$. Therefore the only thresholds in the spectrum of $H_0$ are these $\cE_n$'s. In Theorem \ref{T2} we basically prove that near each $\cE_n$, for any $n\in \N$, there exists a resonance. We also obtain the asymptotic behavior of this pole in terms of $\delta$, where this time the main term is not linear but quadratic.

 As in \cite{BonBruRai07,BruMirPof18}, our strategy is to use an explicit decomposition of the resolvent of the free operator. However, unlike these articles, the Fourier multipliers of $H_0$ are not explicit when $\beta\neq0$, and so we need to start by studying the resolvent of $H_0$. Here we will use the fact that first band function has a unique non-degenerate minimum, see \cite{BriKovRaiSoc09}. Our strategy adapts near other non-degenerate critical points of the band functions (see Remark \ref{upper_minima}), and probably to more general fibered operators.

The paper is organized as follows. In section \ref{S0} we give the basic definitions while in section \ref{results} we present the main theorems. Then, in section \ref{S1} we begin by studying the meromorphic extensions of the resolvents, both of the unperturbed operator $H_0$ and of the perturbed operator $H_\delta$. Next, we prove Theorem \ref{T1} and Theorem \ref{T2}, in sections \ref{S3} and \ref{S5}, respectively. In section \ref{comments} we consider specific geometries for which we comment and refine our results. In the Appendix \ref{ap} we include some computations, such as an explicit expression for the perturbation $H-H_0$ in term of differential operators.
 
\section{Deformed tube and the Hamiltonian}\label{S0}

Let $\gamma:\R\to\R^3$ be a  smooth  curve 
parametrized by its arc length. Assume  that this curve is equipped with a distinct frame $\{{\bf e_1}(s),{\bf e_2}(s),{\bf e_3}(s)\}\subset\R^3$ giving an orthogonal reference frame at each point $\gamma(s)$ of the curve. If the curvature $$\kappa(s):=||\ddot{\gamma}(s)||$$ is different from zero for all $s\in\R$ it is known that the frame  can be chosen such that it satisfies
$${\bf e_1}(s)=\dot{\gamma}(s),\quad {\bf e_2}(s)=\kappa(s)^{-1}\dot{{\bf e}}_{\bf1}(s),\quad{\bf e_3}(s)={\bf e_1}(s)\times{\bf e_2}(s).$$
This is  called the Frenet frame. If $\kappa$ has compact support or vanishes one can still make such choice meaningful, see \cite{EkKovKre08} for a discussion on how to extend the frame and \cite{KS08} for a different approach when the Frenet frame is not available.

Now, let $\theta:\R\to \R$ be a function with bounded derivative, and define the ``general'' frame $\{{\bf e_1^\theta},{\bf e_2^\theta},{\bf e_3^\theta}\}$ which is obtained by rotating the vectors ${\bf e_2}(s),{\bf e_3}(s)$ in an angle $\theta(s)$ around ${\bf e_1}(s)$. It is possible to see that these new vectors satisfy the differential equation 
$$\displaystyle{\frac{d}{ds}}\left(\ba{c}{\bf e_1^\theta}(s) \\{\bf e_2^\theta}(s)\\ {\bf e_3^\theta}(s) \ea \right)=\left(\ba{ccc}0&\kappa(s)\cos\theta(s)&\kappa(s)\sin\theta(s)\\
-\kappa(s)\cos\theta(s)&0&(\tau(s)-\dot{\theta}(s))\\
-\kappa(s)\sin\theta(s)&-(\tau(s)-\dot{\theta}(s))&0
\ea\right)\left(\ba{c}{\bf e_1^\theta}(s) \\{\bf e_2^\theta}(s)\\ {\bf e_3^\theta}(s) \ea \right).$$
Note that ${\bf e_1^\theta}={\bf e_1^0}={\bf e_1}$ and the torsion $\tau$ of the curve is actually defined by the last equation when $\theta\equiv0$.

 Let $\omega$ be a bounded connected region in $\R^2$,  with $C^2$ boundary.  We use the coordinates $t=(t_2,t_3)$ for this region, and define the function 
$\mathcal{L}:\R\times\omega\to \R^3$ given by 
\bel{3}\mathcal{L}(s,t_2,t_3)=\gamma(s)+t_2{\bf e_2^\theta}(s)+t_3{\bf e_3^\theta}(s).\ee

Then we define our twisted-bent tube as the image $\Omega:={\mathcal L}(\R\times \omega)$. As usual, we say that the tube $\Omega$: 
\begin{equation}\label{2}\ba{ll} \mbox{is {bent}} & \mbox{if} \,\, \kappa \,\,\mbox{ is not identically zero};\\
\mbox{is twisted} &\mbox{if} \, \,\tau-\dot{\theta} \,\,\mbox{ is not identically zero}.\ea\ee

Assume that $\cL$ is injective and $||\kappa||_{L^\infty}\sup_{t\in\omega}|t|<1$. This implies that $\cL$ is a global diffeomorphism between the straight tube $\R\times\omega$ and $\Omega$ \cite{EkKovKre08}.

In this article we study a particular type of twisting, which is given by $\dot{\theta}(s)=\beta+\varepsilon(s)$, where $\beta\in\R$ and $\varepsilon$ decay to infinity. Notice that if $\beta\neq0$ and $\tau=\varepsilon\equiv0$ then the tube is periodically twisted. 

\subsection{The Hamiltonian}

 We will study the operator $-\Delta_\Omega$ acting in $L^2(\Omega)$, with Dirichlet boundary condition, i.e., defined through the closed quadratic form 
$$Q[u]=\int_\Omega |\nabla u|^2, \quad u\in\rm{H}_0^1(\Omega).$$ 
The domain of this operator is $\rm{H}^2(\Omega)\cap\rm{H}_0^1(\Omega)$, which follows adapting the proof of \cite[Proposition 2.1]{BriKovRai14}.
 The differential expression of this operator in the global coordinates defined above can be obtained as follows. First, recall that $\cL$ is a diffeomorphism and define $G=(G_{ij})$, as the metric tensor induced by $\cL$, \emph{i.e.,\ } $G_{ij}=(\partial_i\cL)\cdot(\partial_j\cL)$, where for brevity, we sometimes adopt the notations $\partial_1=\partial_s$, $\partial_2=\partial_{t_2}$, $\partial_3=\partial_{t_3}$. Then, following \cite{Kre07} we can set 
$$\ba{ll}h(s,t)&= 1-\kappa(s)\big(t_2\cos\theta(s)+t_3\sin\theta(s)\big)\\
h_2(s,t)&=-t_3\big(\tau(s)-\dot{\theta}(s)\big)\\[.4em]
h_3(s,t)&=t_2\big(\tau(s)-\dot{\theta}(s)\big),\ea$$ 
so  $G$ is explicitly given by
$$G=(G_{jk})=\left(\ba{ccc}h^2+h_2^2+h_3^2&h_2&h_3\\h_2&1&0\\h_3&0&1\ea\right),$$ and $h=\sqrt{\det G}$ (hence $L^2(\Omega)\cong L^2(\R\times\omega, h )$). Also, the operator $-\Delta_\Omega$ is defined in the global coordinates $(s,t)$ acting by 
$$\frac{1}{h}\sum_{j,k=1}^3{\partial_j}hG^{jk}\partial_k,$$
where as usual $(G^{jk}):=G^{-1}$. \ 

Furthermore, it is convenient to study the operator $-\Delta_\Omega$, after a unitary transformation, in $L^2(\R\times\omega)$. Hence we define $U:L^2(\Omega)\to L^2(\R\times \omega)$ by
$$U(f)=(f\circ\cL) h^{1/2}.$$
The self-adjoint operator $U(-\Delta_\Omega) U^*$
 acts on its domain, $D(H)=\rm{H}^2(\R\times\omega)\cap\rm{H}_0^1(\R\times\omega)$, by \bel{7sep18}
-\partial_2^2-\partial_3^2-\frac{\kappa^2}{4h^2}-\left(h^{-1/2}(\partial_s+(\tau-\dot{\theta})\partial_\varphi) h^{-1/2} \right)^2,\ee
where $\partial_\varphi=t_2\partial_3-t_3\partial_2$. We will denote the operator \eqref{7sep18} by $H$. 

\subsection{Reference operator}
 In \eqref{7sep18} consider the operator corresponding to bending zero and periodic twisting, that is, define the operator in $D(H)$ acting by 
 \bel{7sep18a}{H}_0=
 -\partial_2^2-\partial_3^2-\left(\partial_s - \beta\partial_\varphi\right)^2.\ee
 Using the partial Fourier transform with respect to the variable $s$, denoted by $\mathcal{F}$, we can see that $H_0$ is unitary equivalent to an analytically fibered operator: 
 \bel{jan1}\cF {H}_0\cF^{-1}=\int^\oplus_\R {\mathfrak h}_0 (p)dp,\ee
 where for any $p\in\R$
 $${\mathfrak h}_0(p)=-\Delta_\omega-( ip - \beta\partial_\varphi )^2$$
 is an operator acting in $\rm{H}^2(\omega)\cap\rm{H}_0^1(\omega)$. The family $\{{\mathfrak h}_0(p)\}$ is analytic of type A \cite[Lemma 5.1]{BriKovRai14}. Moreover, each operator ${\mathfrak h}_0(p)$ has purely discrete spectrum. Denote by $\{E_n(p)\}$ the increasing sequence of eigenvalues of ${\mathfrak h}_0(p)$. We can see that the spectrum of $H_0$ is purely absolutely continuous and given by 
 $$\sigma(H_0)=\bigcup_{n=1}^\infty E_n(\R)=[\mathcal{E}_1,\infty),$$
 where for $n\geq1$ \bel{eigen}{\mathcal{E}_n:=\min_{p\in\R}E_n(p).}\ee 
In particular  $\mathcal{E}_1$ is the first eigenvalue of ${\mathfrak h}_0 (0)= -\Delta_\omega-\beta^2 \partial_\varphi^2$, and the existence of an effective mass has been proved in \cite{BriKovRaiSoc09}, so we have
 \bel{1bis}E_1(p)=\mathcal{E}_1+m_\beta p^2(1+O(p)),  \qquad \mbox{as }  p \rightarrow 0,  \ee
with $0 <m_\beta\leq 1$, the equality holding (i.e., without the remainder term) if $\beta=0$, with $m_0=1$.

\subsection{Stability of the essential spectrum}

Suppose that $\varepsilon, \kappa, \tau: \R\to \R$ are non-zero functions of class $C^2$ with exponential decay i.e., for some $\alpha>0$
they satisfy
\bel{hypeps}
\kappa(s), \tau(s), \varepsilon(s)=O(e^{-\alpha\langle s\rangle}),
\ee
where $\langle s\rangle:=(1+s^2)^{1/2}$, and the same assumption holds for their first and second derivatives. Straightforward computations show that 
\bel{29nov18}H-H_0\equiv W= f_{0,0}+f_{1,0}\partial_s+f_{0,1}\partial_\varphi+f_{1,1}\partial_s\partial_\varphi+f_{2,0}\partial_s^2+f_{0,2}\partial_\varphi^2,\ee
 where $ f_{i,j}\in C^0(\R\times\omega)$ satisfy 
\bel{decayf}
f_{i,j}(s,t)=O(e^{-\alpha\langle s\rangle}),
\ee
 uniformly with respect to $(s,t) \in \R\times\omega$. 
 The full expression of the $f_{i,j}$ is provided in the Appendix \ref{ap1}. This decomposition allows us to repeat the argument of \cite[Section 4.1]{BriKovRaiSoc09} in order to conclude that $H^{-1}-H_{0}^{-1}$ is compact, and therefore $H$ and $H_{0}$ have the same essential spectrum, that is $$\sigma_{\ess}(H)=[\mathcal{E}_1,\infty).$$

\section{Main results}\label{results}

We want to make our deformed tube depending on a coupling constant. For $\delta\in [0,1]$ consider
$$\dot{\theta}_\delta=\beta+\delta\varepsilon, \quad \kappa_\delta=\delta\kappa \quad \mbox{ and } \quad \tau_\delta=\delta\tau.$$
We define in $D(H)$ the operator given by particularizing \eqref{7sep18} to $\theta_\delta,\kappa_\delta$ and $\tau_\delta$:
\bel{29nov18bis}
H_\delta=-\partial_2^2-\partial_3^2-\frac{\delta^2\kappa^2}{4h_\delta^2}-\left(h_\delta^{-1/2}(\partial_s+\left(\delta(\tau-\varepsilon)-\beta\right)\partial_\varphi)h_\delta^{-1/2}\right)^2,\ee
where $h_{\delta}(s,t)=1-\delta\kappa(s)\big(t_2\cos\theta_\delta(s)+t_3\sin\theta_{\delta}(s)\big)$. Thus, $\delta$ encodes the turning on of the perturbation given by $\theta_\delta,\kappa_\delta$ and $\tau_\delta$. It is clear that $H_1=H$ and the notation is coherent for $H_0$. Note also that by the previous comments we have $\sigma_{\ess}(H_\delta)=[\mathcal{E}_1,\infty)$, and the discrete eigenvalues of $H_\delta$ below $\cE_1$, are the poles of
\bel{6mars19}
z \mapsto \left((H_\delta- z)^{-1}: L^{2}(\R\times\omega)\to D(H)\subset  L^{2}(\R\times\omega)\right)
\ee
in $\C \setminus [\cE_1, + \infty)$. 

In order to define resonances and to state our first main result, let us settle some notations. Let $\eta$ be an exponential weight of the form 
\bel{E:defeta}
\eta(s)=e^{-N \left\langle s\right \rangle} \ \ \mbox{with} \ \ 0<N<\frac\alpha2.
\ee Set $\C^+:= \{k \in \C; \; {\rm Im \,} k >0 \}$ and $\C^{++}:= \{k \in \C; \; {\rm Im \,} k >0; \,\,{\rm Re \,} k >0 \}$. Obviously, taking $z= \cE_{1} + k^2$ 
in \eqref{6mars19}, we obtain that the function
$$ k\mapsto \left((H_\delta-\cE_{1}-k^2)^{-1}:\eta L^{2}(\R\times\omega)\to\eta^{-1}D(H)\subset\eta^{-1} L^{2}(\R\times\omega)\right)$$
is meromorphic on $\C^+$, with poles at $k=i \sqrt{\cE_{1} - \mu} \in i \R^+$, $\mu \in \sigma_{\disc}(H_\delta)$. Moreover, thanks to the weight $\eta$ this operator valued function admits a meromorphic extension to a neighborhood of $k=0$:

\begin{theorem}
\label{T1}Let $\varepsilon, \kappa, \tau: \R\to \R$ be non-zero $C^2$-functions satisfying \eqref{hypeps} and fix a sufficiently small neighborhood of zero $\mathcal{D}$ {in $\C$}.
Then, there exists $\delta_0>0$ such that for $\delta\leq\delta_0$, the analytic operator-valued function 
$$\C^{++}  \ni k \mapsto  \left((H_\delta-\cE_{1}-k^2)^{-1}:\eta L^{2}(\R\times\omega)\to\eta^{-1} D(H)\right)$$ 
admits a meromorphic extension on $\mathcal{D}$. This function has a unique pole $k(\delta)$ in $\cD$, which has multiplicity one and satisfies \bel{2_7_18} 
k(\delta)= i \mu_1\delta +O(\delta^2),\ \ \mu_1\in\R .\ee Moreover, if $\psi_1$ is the normalized function satisfying $
 {\mathfrak h}_0(0)\Bk\psi_1 = \cE_1 \psi_1$, setting 
\bel{D:F}
 F(s)=\int_\omega\left(  \vert \partial_\varphi  \psi_1(t)\vert^2+\frac14 \vert \psi_1(t)\vert^2\right)\left(t_2\cos(\beta s)+t_3\sin(\beta s)\right)d t,
\ee
we have
\bel{12sep18}\mu_1=\frac{\beta||{\partial_\varphi\psi_1}||^2}{\sqrt{m_{\beta} }}
 \int_{\R} ( \tau-\varepsilon)(s) d s-\frac{\beta^2}{\sqrt{m_{\beta} }}\int_{\R} \kappa (s)F(s) d s  .\ee
Further, the pole $k(\delta)$ is a purely imaginary number. \end{theorem}

 With respect to the variable $z$, the previous result means that $$z \mapsto \left((H_\delta- z)^{-1}:  \eta L^{2}(\R\times\omega)\to\eta^{-1} L^{2}(\R\times\omega)\right)$$ admits a meromorphic extension in a neighborhood of $z=\cE_1$ on a $2$-sheeted Riemann surface where the function $z \mapsto \sqrt{z-\cE_1}$, 
 $\sqrt{-1}=i$,  admits an analytic extension. 

\begin{remark}
In the following,  since $k(\delta)$ is purely imaginary, for $\delta$ small it will  be \Bk important to consider the sign of $\mu_1$. If it is positive, then $\cE_1 + k(\delta)^2\in \cE_1 + \R_-$ corresponds to an eigenvalue of $H_\delta$ under the bottom of the essential spectrum. On the other side, if  $\mu_1$ is negative then the resonance  actually lies in the second sheet of the Riemann surface. Such resonance is sometimes called  {\it antibound state} (see \cite[Chap. XI.8.F]{RS3}).
\end{remark}

\begin{remark}\label{upper_minima}
Suppose that for $p_*\in\R$, the number $E_n(p_*)=\cE_*$ is a nondegenerate minimum of the band function $E_n$, such that $\cE_*$ has multiplicity $m_*$
as eigenvalue of the operator $- \Delta_\omega -(ip_* - \beta \pf)^2$,  and such that $\cE_*$ is not a local minimum for all $E_m$ with $m\neq n$. Then, using the same ideas of the proof of the Theorem \ref{T1}, it is possible to  show that there exists a neighborhood of $\cE_*$ such that inside this neighborhood there are at most $m_*$ resonances, and the dependence on $\delta$ is of the form \eqref{2_7_18}. \end{remark} 

The ideas of Remark \ref{upper_minima} can be made more explicit in the case $\beta=0$. We decided to present this case in detail for the following reasons: the asymptotic behavior is generically different from the case $\beta\neq0$; the existence of discrete eigenvalues for asymptotically straight tubes has been studied many times before since \cite{DuEx95}; 
despite this fact there are few results concerning  resonances or the upper thresholds in the spectrum of $\sigma(H_0)$. 

To begin with, notice that for $\beta=0$ the operator $H_0$ is just the Laplacian defined in the straight tube, i.e., \bel{21jan19e}H_0=-\Delta_{\R\times\omega}=D^2_s\otimes I_t-I_s\otimes\Delta_\omega,\ee
where we are  using the notation $D_s=-i\partial_s$. Let $\cE_n$, $n\geq1$, be the increasing sequence of eigenvalues of  $-\Delta_\omega$, counted with multiplicity\footnote{This notation is in agreement with the notation introduced in \eqref{eigen}.}. \Bk  An important difference with the case $\beta\neq0$ is that this time, it is easy to see that the set of  thresholds in the spectrum of $\sigma(H_0)$ correspond precisely with $\{\cE_n\}_{n=1}^\infty$. 
Also, let  $\{\psi_n\}_{n=1}^\infty$  be an  orthonormal basis of eigenvectors of $-\Delta_\omega$  in $L^2(\omega)$.

Our next theorem is the corresponding of Theorem \ref{T1} for $\beta=0$ and, generically, any threshold $\cE_n$. 

\begin{theorem}
\label{T2}Let $\varepsilon, \kappa, \tau: \R\to \R$ be non-zero $C^2$-functions satisfying \eqref{hypeps} and fix a sufficiently small neighborhood of zero $\mathcal{D}$ in $\C$. Suppose $\cE_n$ is a non-degenerate eigenvalue of $-\Delta_\omega$. Then, there exists $\delta_0>0$ such that for $\delta\leq\delta_0$, the analytic operator-valued function $$ k\mapsto \left((H_\delta-\cE_{n}-k^2)^{-1}:\eta L^{2}(\R\times\omega)\to\eta^{-1} L^{2}(\R\times\omega)\right)$$ 
admits a meromorphic extension on $\mathcal{D}$. This function has a unique pole $k_n(\delta)$ in $\cD$, which has multiplicity one. Moreover \bel{19dic18} 
k(\delta)= i \mu_{2,n}\delta^2 +O(\delta^3),\ee
where 
\bel{19dic18a}\ba{lll}\mu_{2,n}&=&\displaystyle{\frac{ 1}{8}\sum_{q\neq n}(\mathcal{E}_{q}-\mathcal{E}_{n})^2 \langle\psi_{q}|t_2\psi_{n}\rangle^2 \langle{\kappa}| (D_{s}^2+\mathcal{E}_{q}-\mathcal{E}_{n})^{-1}{\kappa}\rangle}
\\ &&\displaystyle{-\frac{ 1}{2}\sum_{q\neq n}(\mathcal{E}_{q}-\mathcal{E}_{n}) \langle\psi_{q}|\pf\psi_{n}\rangle^2 \langle (\tau-\varepsilon)|(D_{s}^2+\mathcal{E}_{q}-\mathcal{E}_{n})^{-1}(\tau-\varepsilon)\rangle}\\
&&\displaystyle{-\frac{ 1}{2}\sum_{q\neq n}(\mathcal{E}_{q}-\mathcal{E}_{n}) \langle\psi_{q}|\partial_{\varphi}\psi_{n}\rangle\langle\psi_{q}|t_2\psi_{n}\rangle \langle\tau-\varepsilon| (D_{s}^2+\mathcal{E}_{q}-\mathcal{E}_{n})^{-1}\dot{\kappa}\rangle}.\ea\ee
\end{theorem}

\begin{remark}
In the particular case $n=1$ and if $\mu_{2,1}>0$, this result was previously obtained in \cite{Gru05}. Further, a similar result was obtained in \cite{BruMirPof18} when only twisting is considered, i.e., $\kappa=0$.

\end{remark}

 \begin{remark}
In formula \eqref{19dic18a} when $q < n$, the operator $(D_s^2+\cE_q-\cE_n)^{-1}$ has to
be understood as the limit of $(D_s^2+\cE_q-\cE_n-k^2)^{-1}$, acting in some weighted spaces, when $k\to0$.
\end{remark} 

\begin{remark} We decided to assume that $\cE_n$ is a non-degenerate eigenvalue of $\Delta_\omega$ in order to keep the proof simpler.  As we noticed in Remark \ref{upper_minima} the extension of the resolvent is also possible in the degenerate case, the number of resonances is bounded by the index of degeneracy and the behavior of each resonance is of the type \eqref{19dic18} (see \cite{BruMirPof18} for the case  $\kappa=0$). 
\end{remark}

\section{Meromorphic extension of the resolvents}\label{S1}

In this section, we show that the resolvent of the perturbed operator, acting in weighted spaces, can be extended meromorphically in a neighborhood of $\cE_1$. Our strategy is to exploit an explicit description of the pole of the resolvent of $H_0$ and then conclude by relating it to the resolvent of $H_\delta$ via a resolvent identity.

We start by setting some notations.  For $n\in\N$ and $p \in \R$, denote by $\gp_n(p)$ the orthogonal projection onto ker$({\mathfrak h}_0(p)-E_n(p))$. Let $\Psi_n(\cdot,p)$ be such that $${\mathfrak h}_0(p)\Psi_n(\cdot,p)=E_n(p)\Psi_n(\cdot,p),\quad ||\Psi_n(\cdot,p)||_{L^2(\omega)}=1.$$ Since the first eigenvalue is non-degenerate $\gp_1(p)=|\Psi_1(\cdot,p)\rangle\langle\Psi_1(\cdot,p)|$. Moreover, the functions $$\Psi_1(\cdot,p)$$ can be chosen analytically dependent on $p$. 

In addition, using the unitary operator of complex conjugation we can see that the band functions $E_n(p)$ are even for any $n\geq1$, and for the associated eigenfunctions $\Psi_n(\cdot,p)$, we have
\bel{18jul18}\Psi_n(\cdot,-p)=\overline{\Psi_n(\cdot,p)}, \quad p\in\R.\ee

Since we will mainly need these quantities for $p=0$,  we will use the simplified notation  $\psi_{n}=\Psi_{n}(\cdot,0)$ and $\pi_{n}=\gp_{n}(0)$ (Note that this is also coherent with the notation introduced just before
Theorem \ref{T2}  where $\beta =0$ ) . \Bk

 For $k\in \C^{++}$ set 
$$A(k):=\eta(H_0-\mathcal{E}_1-k^2)^{-1} \one_{(-\infty,\cE_{2})}(H_0)\eta$$
and 
\bel{20jul18_2}B(k):=\eta(H_0-\mathcal{E}_1-k^2)^{-1}\eta-A(k).\ee
Clearly, the operator-valued function $k\mapsto B(k):L^2(\R\times\omega)\to D(H)$ admits an analytic extension on $k$, provided that $k^2 \in \C\setminus [\cE_{2}-\cE_1, + \infty)$, in particular near $k=0$.
For $A$, we have the following meromorphic extension. 

\begin{proposition}\label{mero_part} There exists a neighborhood $\cD$ of zero {in $\C$} in which the operator valued function $k \mapsto A(k)$ admits a meromorphic extension on $\cD$. This extension has a unique pole at $k=0$ and  has  multiplicity one. Further,  in  the Laurent expansion $A(k)=\sum _{l\geq-1}A_lk^l$ we have 
\bel{20jul18}
A_{-1}=\frac{i}{2\sqrt{m_{\beta} }}|\eta\otimes\psi_1\rangle\langle \eta\otimes\psi_1|.\ee
\end{proposition}
\begin{proof}
We start by noticing that because of \eqref{1bis} and for  $|p|$ sufficiently small, say $|p|<\epsilon$, we can write 
 \bel{19jul18}\sqrt{E_1(p)-\mathcal{E}_1}=pd(p),\ee
 where $d(p)=m_\beta^{1/2}+O(p)$, is analytic for $|p|<\epsilon$.
 
Without loss of generality, consider $\epsilon< N$ and let $\cD\subset\C$ be an open subset of $B(0,\epsilon)$ containing zero and within  the bounded part defined by the curve $zd(z)$, where $|z|=\epsilon$. Set also $\epsilon_1$ such that $E_1(\epsilon_1)=\cE_2$. By the parity of the band functions we have that $E_1\left((-\epsilon_1,\epsilon_1)\right)=[\cE_1,\cE_2)$.

Now, note that $A(k)=\eta\cF^{-1}\int_{(-\epsilon_1,\epsilon_1)}^\oplus(E_1(p)-\mathcal{E}_1-k^2)^{-1}\gp_1(p)dp\cF\eta$, then its integral kernel is given by 
$$
\displaystyle{\frac{\eta(s)
\eta(s')}{2\pi}\int_{(-\epsilon_1,\epsilon_1)}e^{ip(s-s')}\Psi_1(t,p)\Psi_1(t',p)(E_1(p)-\mathcal{E}_1-k^2)^{-1}dp}.$$
Take $\epsilon<\epsilon_1$.
For $k\in\cD \cap \C^{++}$, using \eqref{19jul18} and the analytic properties of $\Psi_1$ with respect to $p$, we have 
$$
\ba{ll}
&\displaystyle{\int_{[-\epsilon,\epsilon]}e^{ip(s-s')}\Psi_1(t,p)\Psi_1(t',p)(p^2d(p)^2-k^2)^{-1}dp}\\[.8em]
 =&\displaystyle{\frac{1}{2 k}\int_{[-\epsilon,\epsilon]}e^{ip(s-s')}\Psi_1(t,p)\Psi_1(t',p)\left(\frac{1}{pd(p)-k}-\frac{1}{pd(p)+k}\right)dp}\\[.8em]
 =&\displaystyle{\frac{1}{2 k}}\left(
 -\displaystyle{\int_{\gamma_1}\frac{e^{iz(s-s')}\Psi_1(t,z)\Psi_1(t',z)}{zd(z)-k}dz}
+\displaystyle{\int_{\gamma_2}\frac{e^{iz(s-s')}\Psi_1(t,z)\Psi_1(t',z)}{zd(z)+k}dz } \right)
 \ea$$ where we have used the Cauchy formula for the 
 curves $\gamma_1(t)=\epsilon e^{-i t}$, $0<t<\pi$, 
and $\gamma_2(t)=\epsilon e^{i t}$, $0<t<\pi$. Consider the operator with integral kernel depending on $k$ given by 
$$\eta(s)
\eta(s') \left(
 -\displaystyle{\int_{\gamma_1}\frac{e^{iz(s-s')}\Psi_1(t,z)\Psi_1(t',z)}{zd(z)-k}dz}
+\displaystyle{\int_{\gamma_2}\frac{e^{iz(s-s')}\Psi_1(t,z)\Psi_1(t',z)}{zd(z)+k}dz } \right) .$$ For every $k\in\cD$, this operator is bounded from $L^2(\R\times\omega)$ to $D(H)$ and due to our choice of $\epsilon$ and $\cD$ it depends analytically on $k$. This implies that $A(k)$ has a meromorphic extension from $\cD \cap \C^{++}$ to $\cD$ with a unique pole at zero (the integral over $(-\epsilon_1,\epsilon_1)\setminus[-\epsilon, \epsilon]$ is clearly analytic).

 Finally, 
to compute $A_{-1}$ we just need to evaluate the above line integral at $k=0$, obtaining the integral kernel 
$$\frac{\eta(s)
\eta(s')}{4\pi} \displaystyle{\int_{\gamma}\frac{e^{iz(s-s')}\Psi_1(t,z)\Psi_1(t',z)}{zd(z)}dz}=\frac{ i}{2 \sqrt{m_{\beta} }} \eta(s)\eta(s')\Psi_1(t,0)\Psi_1(t',0)$$
(here $\gamma=e^{i t}$, $0<t\leq2\pi$). This kernel obviously defines a rank one operator, in consequence the pole at zero has multiplicity one. 
\end{proof}

From the previous Proposition and \eqref{20jul18_2}, we can see that 
\bel{6sep17a}\eta(H_0-\mathcal{E}_1-k^2)^{-1}\eta=
\frac{{A_{-1}}}{k}+F(k),\ee
where $F(k):=\sum_{l=0}^\infty A_lk^l+B(k)$ can be extended analytically from $k\in \cD \cap \C^{++}$ to $k\in\cD$, in $\cL(L^2(\R\times\omega), D(H))$.

We will denote the meromorphic extension of $\eta(H_0-\mathcal{E}_1-k^2)^{-1}\eta$ to $\cD$ by $\eta R_{0}(k)\eta$.

Let us now  consider the problem of extending the resolvent of $H_\delta $. Set 
\bel{7sep18b} W_\delta=H_{\delta}-H_{0}.\ee

Then, as in \eqref{29nov18}, we have the  decomposition 
$W_\delta= \sum_{0\leq i,j\leq 2} f_{i,j}(\delta)\partial_{i,j},$ where $ f_{i,j}(\delta)(s,t)=O(\delta e^{-\alpha\langle s\rangle}). $
This, combined with  \eqref{E:defeta}, ensures that \bel{11}\eta^{-1}W_\delta\eta^{-1}\in\cL(D(H),L^2({\R\times\omega})).\ee
Moreover, $W_\delta\psi_n$ is well defined and  satisfies 
\bel{31jan19}
(W_\delta\psi_n)(s,t)= \delta e^{-\alpha\langle s\rangle}(\tilde{g}+O(1)),
\ee
uniformly with respect to $(s,t) \in \R\times\omega$ and with $\tilde{g}$ bounded and independent of $\delta$ (see the Appendix \ref{ap1}).

Next, consider the following identity valid for $k\in\C^{++}$
\bel{4sept17e}
\eta(H_{\delta}-\mathcal{E}_1-k^2)^{-1} \eta=\eta R_{0}(k)\eta\left({\rm Id}+\eta^{-1} W_\delta\eta^{-1}\eta R_{0}(k)\eta\right)^{-1}.\ee
 This identity motivates the following preparatory Lemma for which we need to define 
\bel{E:defPhi}\Phi_\delta:= \frac{i}{2\sqrt{m_{\beta} }}
\eta^{-1}W_\delta \psi_1.
\ee 
 \begin{lemma}\label{Lholo0}
There exists a neighborhood $\cD$ of zero and $\delta_0>0$ such that 
for any $0<\delta\leq\delta_0$ and $k\in \cD\setminus\{0\}$ 
$$\eta^{-1} W_\delta\eta^{-1}\eta R_{0}(k)\eta=\frac{1}{k} K_0(\delta) + T(\delta,k),$$
where $K_0(\delta)$ is the rank one operator \bel{4sep17} K_0(\delta)= |\Phi_\delta\rangle\langle\eta\otimes\psi_1|,
\ee 
and 
$\cD\ni k\mapsto(T(\delta,k)$: $L^2({\R\times\omega})\to L^2({\R\times\omega}))$ is an analytic operator-valued function. Moreover, 
\bel{2sep17}\sup_{k \in \cD}||T(\delta,k)||=O(\delta).
\ee
\end{lemma}
\begin{proof}
Taking into account \eqref{6sep17a} define \bel{5oct18}K_0(\delta):=\eta^{-1}W_\delta\eta^{-1}A_{-1}\ee and \bel{5oct18b}T(\delta,k):=\eta^{-1}W_\delta\eta^{-1}F(k).\ee Thanks to \eqref{31jan19} and \eqref{20jul18}, $K_0(\delta)$ is well defined and satisfies \eqref{4sep17}. Then the Lemma follows from   \eqref{11} and  
$F(k)\in\cL(L^2(\R\times\omega),D(H))$ (see after \eqref{6sep17a}) . \end{proof}

 Now we are in condition to prove the main result of this section, namely the existence of a meromorphic extension of the resolvent of the full Hamiltonian in a neighborhood of the bottom of its spectrum. 

\begin{proposition}\label{prhdelta}
There exists a neighborhood $\cD$ of zero {in $\C$},  and $\delta_0>0$ such that for $\delta<\delta_0$ the operator valued function $k \mapsto \eta (H_\delta-\cE_1-k^2)^{-1}\eta$ admits a meromorphic extension on $\cD$. We write this extension by $\eta R(k)\eta $.
\end{proposition}
\begin{proof} Consider the identity \eqref{4sept17e}, and note that from Lemma \ref{Lholo0} for $k\in \mathcal{D}\setminus\{0\}$ and 
 \ $\delta$ sufficiently small we can write 
\bel{27sep17}\Big( {\rm Id} + \eta^{-1} W_\delta R_{0}(k)\eta \Big)=
\Big( {\rm Id} +T(\delta,k) \Big)
\Big( {\rm Id} +\frac{1}{k}({\rm Id}+ T(\delta,k) )^{-1} K_0(\delta) \Big).\ee
For $k\in \cD\setminus\{0\}$ let us set the rank one operator
$$K:=\frac{1}{k}({\rm Id}+ T(\delta,k) )^{-1} K_0(\delta)=\frac{1}{k} |({\rm Id}+ T(\delta,k) )^{-1}{\Phi}_{\delta}\rangle\langle \eta\otimes\psi_{1}|,$$ 
and define  \bel{defnu}
 \nu_\delta(k):=\langle\eta\otimes\psi_1|({\rm Id}+ T(\delta,k))^{-1}\Phi_\delta\rangle. \ee Note that if $\nu_\delta(k)\neq0$, then $\frac{\nu_{\delta}(k)}{k}$ in the only non zero eigenvalue of $K$. The inverse of $({\rm Id}+K)$ is of the form 
\bel{29sept17a}
({\rm Id}+K)^{-1}=\Pi_{\delta}^{\perp}+\frac{k}{k+\nu_{\delta}(k)} \Pi_{\delta},\ee
where $\Pi_{\delta}^{\perp}$ is the projection onto $(\mbox{span }\{\eta\otimes\psi_{1}\})^{\perp}$ into the direction $\tilde{\Phi}_\delta$ and $\Pi_{\delta}={\rm Id}-\Pi_{\delta}^{\perp}$. 

Define $k(\delta)$ as the solutions of $k+\nu_\delta(k)=0$. In consequence, putting together \eqref{4sept17e}, \eqref{6sep17a}, \eqref{27sep17} and \eqref{29sept17a} we obtain that for all $k\in \cD\setminus \{ 0, k(\delta)\}$
$$\eta (H_\delta-\cE_1-k^2)^{-1}\eta= \Big( \frac{1}{k} A_{-1}+F(k) \Big)
\Big( \Pi_{\delta}^{\perp} + \frac{k}{k+\nu_{\delta}(k)} \Pi_{\delta} \Big) \; ({\rm Id}+ T(\delta,k) )^{-1}.$$
By the definition of $\Pi_{\delta}^{\perp}$ and $A_{-1}$, 
we have that $ A_{-1} \Pi_{\delta}^{\perp}=0$ and then:
\bel{12feb19}\eta (H_\delta-\cE_1-k^2)^{-1}\eta = \frac{1}{k+\nu_\delta(k)} \Big(A_{-1}+kF(k)\Big)\Pi_{\delta} ({\rm Id}+ T(\delta,k) )^{-1}
 + F(k) \Pi_{\delta}^{\perp} ({\rm Id}+ T(\delta,k) )^{-1}.\ee
 Therefore, for $\delta$ sufficiently small, by the analyticity of $F(k)$ and Lemma \ref{Lholo0},   $k \mapsto \eta (H_\delta-\cE_1-k^2)^{-1}\eta$ admits a meromorphic extension to $\cD$.
\end{proof}

\section{Proof of theorem \ref{T1} ( $\beta\neq0$)}\label{S3}

In this section, we provide the proof of Theorem \ref{T1}. We start by collecting some information about the behavior of the quantity  $\nu_\delta$ defined in \eqref{defnu}. First, in Lemma \ref{le22may18}, we derive its asymptotic behavior while in Lemma \ref{lereal_part} we obtain a property that will allow us to show that the pole is purely imaginary. 

\begin{lemma}\label{le22may18}
 Let  $\cD$ and $\delta_0$ be as in Proposition \ref{prhdelta}. Then for  $k\in \cD$  and $\delta \in[0,\delta_0]$ we have:
$$\nu_{\delta}(k)=-i\mu_1\delta -i\mu_{2}\delta^2+k\delta^2g_\delta(k)+O(\delta^3),$$
where $\mu_1$ is given in \eqref{12sep18},
$\mu_{2}$ is a real number and $g_\delta$ is an analytic function.
\end{lemma}
\begin{proof}
First note that by the Neumann series and \eqref{2sep17} 
$$\ba{ll}({\rm Id}+T(\delta,k))^{-1}&={\rm Id}-T(\delta,k)+O(\delta^2)\\
&={\rm Id}-T(\delta,0)+kG(\delta,k) +O(\delta^2),\ea
$$
where $G$ is analytic with uniformly bounded norm in $\delta$ small and $k\in \cD$. 
Then 
\bel{19jan19}\ba{ll}\nu_{\delta}(k)&=\displaystyle{\frac{i}{2\sqrt{m_{\beta}}}\langle\psi_1\eta|\eta^{-1} W_\delta \psi_1\rangle-
\frac{i}{2\sqrt{m_{\beta}}}\langle\psi_1\eta|T(\delta,0)\eta^{-1} W_\delta \psi_1\rangle+k\delta^2g_\delta(k)+O(\delta^3)},
\ea\ee
 where we used \eqref{31jan19} and \eqref{2sep17}. By \eqref{31jan19}, since $T(\delta,0)=\eta^{-1} W_\delta\eta^{-1}F(0)$, to find $\mu_1$ we need only the first asymptotic term of $\langle\psi_1\eta|\eta^{-1} W_\delta\psi_1\rangle$. We therefore have $\mu_1=-\frac{1}{2\sqrt{m_{\beta}}}\breve{\mu}_{1,1}$, where $\breve{\mu}_{1,1}$ is computed in Lemma \ref{le_app} in the Appendix \eqref{ap2}. Further, $\mu_2$ comes  from the second term of $\langle\psi_1\eta|\eta^{-1} W_\delta\psi_1\rangle$ and from the main term of $\langle\psi_1\eta|T(\delta,0)\eta^{-1} W_\delta \psi_1\rangle$ and is hence real.
 \end{proof}

\begin{lemma}\label{lereal_part}
For all $\alpha$ real and small 
$$i \nu _{\delta}(i\alpha)\in\R.$$
\end{lemma}
\begin{proof}
First, we will prove that $F(i\alpha)$ is self-adjoint. To do that recall that $F(k)=A(k)-\frac{A_{-1}}{k}+B(k).$
Since $B(i\alpha)=\eta f(H_0)\eta$ for $f(x)=(x-\cE_1+\alpha^2)^{-1}\one_{[\cE_2,\infty)}(x)$, it is obviously self-adjoint. 
 
From Proposition \ref{mero_part} we have that the kernel of $A(i \alpha)-\frac{A_{-1}}{i \alpha}$ is 
$$\displaystyle{\frac{\eta(s)\eta(s')}{2\pi}\left(\int_{[-\epsilon_1,\epsilon_1]}\frac{e^{ip(s-s')}\Psi_1(t,p)\Psi_1(t',p)}{E_1(p)-\mathcal{E}_1+\alpha^2}dp-
\frac{\pi i  \Psi_1(t,0)\Psi_1(t',0)}{\sqrt{m_{\beta}}i\alpha}\right)},$$
Moreover, making a change of variables for $\alpha\neq0$, using the parity of the band functions and the property \eqref{18jul18} we obtain 
$$\int_{[-\epsilon_1,\epsilon_1]}\frac{e^{ip(s-s')}\Psi_1(t,p)\Psi_1(t',p)}{E_1(p)-\mathcal{E}_1+\alpha^2}dp=\int_0^{\epsilon_1}\frac{2\,{\rm Re}{\big(e^{ip(s-s')}\Psi_1(t,p)\Psi_1(t',p)\big)}}{E_1(p)-\mathcal{E}_1+\alpha^2}dp,$$
which is real and symmetric. Also notice that $\Psi_1(t,0)$ is real valued. This gives us the self-adjointness of 
 $A(i\alpha)-\frac{A_{-1}}{i\alpha}$.

Now, recall that $\nu _{\delta}(k):=\frac{i}{2\sqrt{m_{\beta} }}\langle\psi_1\eta|({\rm Id}+ T(\delta,k))^{-1}\eta^{-1}W_\delta\psi_1\rangle,$ then using the Neumann series it is enough to show that $\langle\psi_1\eta|T(\delta,k)^n\eta^{-1}W_\delta\psi_1\rangle$ is real valued for all $n\in\N$ when $k=i\alpha$. But since $T(\delta,k)=\eta^{-1}W_\delta\eta^{-1}F(k)$, denoting by $\cA:=F(i\alpha)$, and $\cB:=\eta^{-1}W_\delta\eta^{-1}$, we have 
$$\langle\psi_1\eta|T(\delta,k)^n\eta^{-1}W_\delta\psi_1\rangle=\langle W_\delta \psi_1|\eta^{-1} \cA \cB \cA ...\,\cA \cB\cA\eta^{-1}W_\delta\psi_1\rangle,$$ which is real because of the self-adjointness of $\cA$ and $\cB$.
\end{proof}

\begin{proof}[Proof of Theorem \ref{T1}]
The meromorphic extension of the resolvent to $\cD$ was obtained in Proposition \ref{prhdelta}. Let us prove that there is only one pole inside $\cD$. Consider a circle $\gamma$ within $\mathcal{D}$ and take the analytic functions $f(k):=k+\nu_{\delta}(k)$ and  $g(k):=k $.  Then, since the radius of $\gamma$ is fixed, taking $\delta$ small and using  Lemma \ref{le22may18} we have that $$|f-g|<|g|$$ on $\gamma$. Thus, Rouche's Theorem implies the existence of a unique solution of the equation $k+\nu_{\delta}(k)=0$ inside $\gamma$. 
 
The behavior \eqref{2_7_18} follows immediately from Lemma \ref{le22may18}.  The multiplicity of $k(\delta)$ is the rank of the residue of $\eta R_\delta(k)\eta$, which from \eqref{12feb19} is giving by the rank of $(A_{-1}+k(\delta)F(k(\delta)))\Pi_\delta$. By the definition of $\Pi_\delta$ the rank of $(A_{-1}+k(\delta)F(k(\delta)))\Pi_\delta$ is at most one. By \eqref{2_7_18}, for $\delta $ small, this rank is equal to the rank of $A_{-1}\Pi_\delta$ which is one if $\nu_\delta(k)\neq0$. If $\nu_\delta(k)=0$, in the proof of Proposition \ref{prhdelta}, the equality  \eqref{29sept17a} becomes $({\rm Id}+K)^{-1}=\Pi_{\delta}^{\perp}+ \Pi_{\delta}={\rm Id}$. Therefore \eqref{4sept17e} implies that $k=0$ is the unique pole of the resolvent of $H_\delta$, and has multiplicity one by Proposition \ref{mero_part}. 

 Finally, to see that $k(\delta)$ is purely imaginary, consider the real-valued function $t(x)=i( i x+\nu_\delta(ix))$, $x \in \R$ (have in mind Lemma \ref{lereal_part}). Applying Lemma \ref{le22may18}, if $\mu_1\neq0$, we have that $t(0)$ and $t(i2\mu_1\delta)$ have different signs, for $\delta$ small. In consequence, $t$ has a root of modulus smaller than $2|\mu_1 \delta|$. By the uniqueness proved above, this root is $i k(\delta)$. If $\mu_1=0$, again by Lemma \ref{le22may18}, $\nu_\delta(\pm\delta)=O(\delta^2)$, therefore $t(\delta)$ and $t(-\delta)$ have different signs for $\delta$ small. Arguing as before we conclude the proof of the Theorem. \end{proof}

\section{Proof of Theorem \ref{T2} ( $\beta=0$)}\label{beta_zero}\label{S5}
\begin{proof}[Proof of Theorem \ref{T2}] First note that from \eqref{21jan19e}, we have that for $k\in\C^{++}$
\bel{21jan19}
(H_0-\cE_n-k^2)^{-1}=\sum_{q \geq 1} (D^2_{3}+(\cE_q-\cE_n)-k^2)^{-1}\otimes\pi_q.\ee
Thus, the meromorphic extension of $\eta(H_0-\cE_n-k^2)^{-1}\eta$ is explicit. Indeed, from \eqref{21jan19}, the operator $\sum_{q > n} (D^2_{3}+(\cE_q-\cE_n)-k^2)^{-1}\otimes\pi_q=\one_{[\cE_{n+1},\infty)}(H_0)$
has an obvious analytic extension. 
Further, for $q\leq n$ the kernel of $(D^2_{3}+(\cE_q-\cE_n)-k^2)^{-1}$ is explicitly given by
$$\frac{i}{2\sqrt{k^2+\cE_n-\cE_q}}e^{i|s-s'|\sqrt{k^2+\cE_n-\cE_q}}.$$
For $q<n$, these kernels define operators analytically dependent on $k$ in a neighborhood of $0$. For the case $q=n$, the operator $\mathsf{A}_n(k):=\eta(D^2_{3}-k^2)^{-1}\eta\otimes\pi_n$
has a meromorphic extension with a unique pole at zero, which has multiplicity one.
In conclusion, denoting by $B_{n}(k):=\sum_{q\neq n} (D^2_{3}+(\cE_q-\cE_n)-k^2)^{-1}\otimes\pi_q,$ we have the meromorphic extension $$\eta(H_0-\cE_n-k^2)^{-1}\eta=\tfrac{1}{k}\mathsf{A}_{n,-1}+F_n(k),$$ where $\mathsf{A}_{n,-1}$ has integral kernel 
$\frac{i}{2}\eta(s')\eta(s)\psi_n(t')\psi_n(t)$
and $F_n(k)=\mathsf{A}_{n}(k)-\tfrac{1}{k}\mathsf{A}_{n,-1}+B_n(k)$ is analytic in a neighborhood of zero.

Next, the proof of the extension of the resolvent of $(H_\delta-\cE_n-k^2)^{-1}$ to a neighborhood of zero $\mathcal{D}$ works mutatis mutandis the proof of Proposition \ref{prhdelta}. Also, since we are assuming that $\cE_n$ is a non-degenerate eigenvalue of $\Delta_\omega$, the proof of the uniqueness of the pole in $\cD$ works similarly as well. 

 Recalling that if $\beta=0$ we have $m_\beta=1$, we define $$\Phi_{n,\delta}:= \frac{i}{2 } 
\eta^{-1}W_\delta \psi_n, \quad T_n(\delta,k):=\eta^{-1}W_\delta\eta^{-1}F_n(k)$$and $$ 
 \nu_{n,\delta(k)}:=\langle \eta\otimes\psi_n|({\rm Id}+ T_n(\delta,k))^{-1}\Phi_{n,\delta}\rangle. $$
 
As in Lemma \ref{le22may18} we have that 
$\nu_{n,\delta}(k)=-i\mu_{2,n}\delta^2+k\delta^2g_\delta(k)+O(\delta^3)$ (the term of order $\delta$ is zero due to  $\beta=0$ ). To compute $\mu_{2,n}$ we need to consider the second order term of 
 $$\ba{ll}\displaystyle{-\frac{1}{2}\langle \psi_n\eta|\eta^{-1} W_\delta \psi_n\rangle+
\frac{1}{2}\langle \psi_n\eta|T_n(\delta,0)\eta^{-1} W_\delta \psi_n\rangle}=-\frac{1}{2}\langle \psi_n|W_\delta \,\psi_n\rangle+\frac{1}{2} \langle  W_\delta \psi_n|\eta^{-1} F(0) \eta^{-1}W_\delta \psi_n\rangle,\ea$$
 where the first inner product in the right hand side needs to be understand as a duality. First, from Lemma \ref{le_app} in the Appendix 
\bel{20jan19e}-\langle \psi_n|W_\delta \,\psi_n\rangle
=\left(\left\|\frac{\kappa}{2}\right\|^2-\left\|\frac{\dot\kappa}{2}t_2\psi_n+(\tau-\epsilon)\pf\psi_n\right\|^2\right)\delta^2+O(\delta^3).\ee

Now, let us consider the term given by $\langle  W_\delta \psi_n|\eta^{-1} F(0) \eta^{-1}W_\delta \psi_n\rangle.$
The integral kernel of the operator $\left(\mathsf{A}_{n}(k)-\tfrac{1}{k}\mathsf{A}_{n,-1}\right)\big|_{k=0}$ is 
$$-\frac{|s-s'|}{2} \eta(s')\eta(s)\psi_n(t')\psi_n(t).$$
Also, note that from the proof of Lemma \ref{le_app} in the Appendix we obtain \bel{20jan19b}W_\delta \psi_n=\Big(-\frac{1}{2} \ddot{\kappa} t_2\psi_n+(\dot{\epsilon}-\dot{\tau})\pf\psi_n\Big)\delta+O(\delta^2).\ee
In consequence, since $\psi_n$ is orthogonal to $\pf\psi_n$, using \eqref{20jan19b} we obtain
\bel{20jan19}\ba{ll}
\langle W_\delta \psi_n|\left(\mathsf{A}_{n}(k)-\tfrac{1}{k}\mathsf{A}_{n,-1}\right)\big|_{k=0}W_\delta \psi_n\rangle
&=-\delta^2\displaystyle{\ \int_{\R^{ 2}}\frac{|s-s'|}{8} \ddot\kappa(s')\ddot\kappa(s)ds'ds\left(\int_\omega t_2\psi_n(t)^{2}dt\right)^2+O(\delta^3)}\\[1em]&=\displaystyle{\frac{\delta^2}{4} \|\dot\kappa\|^{2}\,\langle \psi_n|t_2\psi_n\rangle^2+O(\delta^3)}\ea
\ee
Further, we compute 
\bel{20jan19a}
\langle  W_\delta \psi_n|\eta^{-1} B_n(0)\eta^{-1}W_\delta \psi_n\rangle
\ee
$$
 =\delta^2\left\langle\big(\frac{1}{2} \ddot{\kappa} t_2\psi_n+(\dot{\tau}-\dot{\epsilon})\pf\psi_n\big)\Big|\sum_{q \neq n} \big((D^2_{s}+(\cE_q-\cE_n))^{-1}\otimes\pi_q\big)\big(\frac{1}{2} \ddot{\kappa} t_2\psi_n+(\dot{\tau}-\dot{\epsilon})\pf\psi_n\big)\right\rangle+O(\delta^3)
$$
$$
 =\delta^2\left\langle\big(\frac{1}{2} \dot{\kappa} t_2\psi_n+({\tau}-{\epsilon})\pf\psi_n\big)\Big|\sum_{q \neq n} \big(I_s\otimes\pi_q\big)\big(\frac{1}{2} \dot{\kappa} t_2\psi_n+({\tau}-{\epsilon})\pf\psi_n\big)\right\rangle+O(\delta^3)
$$
$$
 -\delta^2\left\langle\big(\frac{1}{2} \dot{\kappa} t_2\psi_n+({\tau}-{\epsilon})\pf\psi_n\big)\Big|\sum_{q \neq n} (\cE_q-\cE_n)\big((D^2_{s}+(\cE_q-\cE_n))^{-1}\otimes\pi_q\big)\big(\frac{1}{2} \dot{\kappa} t_2\psi_n+({\tau}-{\epsilon})\pf\psi_n\big)\right\rangle,
$$
where we have integrated by parts in the variable $s$ and then added and subtracted in each term the operator $(\cE_q-\cE_n)\big((D^2_{s}+(\cE_q-\cE_n))^{-1}\otimes\pi_q\big)$. 

The first term in the last equality is equal to 
$$
 \delta^2\|\frac{\dot\kappa}{2}\|^2\sum_{q \neq n} |\langle t_2\psi_n|\psi_{q}\rangle|^2
+\delta^2\|\tau-\varepsilon\|^2\sum_{ q \neq n} |\langle \psi_{ q}|\pf\psi_n\rangle|^2+\delta^2 2 \langle \frac{\dot\kappa}{2}|\tau-\varepsilon\rangle\sum_{ q \neq n} \langle \psi_{ q}|t_2\psi_n\rangle\langle \psi_{ q}|\pf\psi_n\rangle
$$
\bel{20jan19d}=\delta^2\|\frac{\dot\kappa}{2}\|^2\left(\|t_2\psi_n\|^2-\langle \psi_n|t_2\psi_n\rangle\right)+\delta^2\|\tau-\varepsilon\|^2\|\pf\psi_n\|^2+\delta^2 2\langle\frac{\dot\kappa}{2}| \tau-\varepsilon\rangle \langle t_2\psi_n|\pf\psi_n\rangle.
\ee
$$=\delta^2\left(\left\|\frac{\dot\kappa}{2}t_2\psi_n+(\tau-\epsilon)\pf\psi_n\right\|^2-\langle \psi_n|t_2\psi_n\rangle\|\frac{\dot\kappa}{2}\|^2\right)$$
Putting together \eqref{20jan19e}, \eqref{20jan19}, \eqref{20jan19a}, \eqref{20jan19d} we get 
$$\ba{lll}\mu_{2,n}&=&\displaystyle{\frac{1}{8}||\kappa||^2-\frac{1}{2}\sum_{q\neq n}(\mathcal{E}_{q}-\mathcal{E}_{n}) \langle\psi_{q}|\partial_{\varphi}\psi_{n}\rangle^2 \langle (\tau-\varepsilon)|(D_{s}^2+\mathcal{E}_{q}-\mathcal{E}_{n})^{-1}(\tau-\varepsilon)\rangle}
\\&&\displaystyle{-\frac{1}{8}\sum_{q\neq n}(\mathcal{E}_{q}-\mathcal{E}_{n}) \langle\psi_{q}|t_2\psi_{n}\rangle^2 \langle\dot{\kappa}| (D_{s}^2+\mathcal{E}_{q}-\mathcal{E}_{n})^{-1}\dot{\kappa}\rangle}
\\&&\displaystyle{-\frac{1}{2}\sum_{q\neq n}(\mathcal{E}_{q}-\mathcal{E}_{n}) \langle\psi_{q}|t_2\psi_{n}\rangle \langle\psi_{q}|\partial_{\varphi}\psi_{n}\rangle\langle\tau-\varepsilon| (D_{s}^2+\mathcal{E}_{q}-\mathcal{E}_{n})^{-1}\dot{\kappa}\rangle}.\ea$$
To finish, in the second sum integrate by parts in the variable $s$ again, obtaining $$ \sum_{n\geq2}(\mathcal{E}_{q}-\mathcal{E}_{n}) \langle\psi_{q}|t_2\psi_{n}\rangle^2 \langle\dot{\kappa}| (D_{s}^2+\mathcal{E}_{q}-\mathcal{E}_{n})^{-1}\dot{\kappa}\rangle$$$$
=\|\kappa\|^2\sum_{n\geq1}(\mathcal{E}_{q}-\mathcal{E}_{n}) \langle\psi_{q}|t_2\psi_{n}\rangle^2 -\sum_{n\geq2}(\mathcal{E}_{q}-\mathcal{E}_{n})^2 \langle\psi_{q}|t_2\psi_{n}\rangle^2 \langle{\kappa}| (D_{s}^2+\mathcal{E}_{q}-\mathcal{E}_{n})^{-1}{\kappa}\rangle$$
$$
=\|\kappa\|^2\langle(-\Delta_\omega-\cE_1)t_2\psi_n|t_2\psi_n\rangle -\sum_{n\geq2}(\mathcal{E}_{q}-\mathcal{E}_{n})^2 \langle\psi_{q}|t_2\psi_{n}\rangle^2 \langle{\kappa}| (D_{s}^2+\mathcal{E}_{q}-\mathcal{E}_{n})^{-1}{\kappa}\rangle$$
$$
=\|\kappa\|^2\langle-2\partial_2\psi_n|t_2\psi_n\rangle -\sum_{n\geq2}(\mathcal{E}_{q}-\mathcal{E}_{n})^2 \langle\psi_{q}|t_2\psi_{n}\rangle^2 \langle{\kappa}| (D_{s}^2+\mathcal{E}_{q}-\mathcal{E}_{n})^{-1}{\kappa}\rangle$$
$$
=\|\kappa\|^2 -\sum_{n\geq2}(\mathcal{E}_{q}-\mathcal{E}_{n})^2 \langle\psi_{q}|t_2\psi_{n}\rangle^2 \langle{\kappa}| (D_{s}^2+\mathcal{E}_{q}-\mathcal{E}_{n})^{-1}{\kappa}\rangle.$$
this will finally give us \eqref{19dic18a}. \end{proof}

\section{Discussion}\label{comments}

In the statement of Theorems \ref{T1} and \ref{T2} appear several different quantities that determine the constants $\mu_1$ and $\mu_2$. In order to understand better the influence of the curvature and twisting on these constants, we will consider some different cases separately and make some comments about them. 

\subsection{Theorem \ref{T1}, $\kappa=0$} I this situation  we obtain $\mu_1= -\frac{\beta }{\sqrt{m_\beta}}||\partial_\varphi \psi_1||^2\int_\R\varepsilon\,ds,$
which implies that the existence of a bound state or an anti-bound state near $\cE_1$ depends on the sign of $\int_\R\varepsilon(s)\,ds.$ In particular, using that 
$$ 
\delta \varepsilon = \dot{\theta_\delta} - \beta= \frac{\dot{\theta_\delta}^2-\beta^2}{\dot{\theta_\delta}+\beta} = \frac{\dot{\theta_\delta}^2-\beta^2}{2\beta} + O(\delta^2),
$$
it can be easily seen that 
$$k(\delta)=  \frac{-i}{2\sqrt{m_\beta}} ||\partial_\varphi \psi_1||^2\int_\R(\dot{\theta_\delta}^2-\beta^2)\,ds +O(\delta^2),$$  
in agreement with the articles \cite{ExKov05,BrHaKr15}. In particular in \cite{ExKov05}, with a slightly different notation, it is shown that if $\int_\R(\dot{\theta}^2-\beta^2)\,ds\leq 0$ then we obtain eigenvalues while in \cite{BrHaKr15} they prove a Hardy inequality under the assumption $\beta\epsilon\geq 0$ (that correspond to $\int_\R(\dot{\theta}^2-\beta^2)\,ds>0$). What we have shown is that in the last case we have an antibound state.  Moreover, it is possible to compute the term of order $\delta^2$  if we assume for example that $\varepsilon$ is odd. In this case $\mu_1=0$ and
$$\mu_2=\int_\R  \sum_{q=1} \frac{p^2}{E_q(p)-\cE_n}
|\langle \Psi_n(p)|2\beta \hat{E}\pf^2\psi_1-\hat{\varepsilon}\pf\psi_1\rangle|^2dp-\|\pf\psi_1\|^2\int_\R |\hat{\varepsilon}(p)|^2dp,
$$
where $E(s)=\int_{-\infty}^s\varepsilon(r)dr$. 
 However, we were not able to find the sign of $\mu_2$, which would help us to see  whether  \Bk the condition $\int_\R(\dot{\theta}^2-\beta^2)\,ds\geq 0$ is optimal or order to obtain eigenvalues below the bottom of the essential spectrum.

\subsection{Theorem \ref{T1}, $\int\tau-\varepsilon=0$} In this case $\mu_1 =-\frac{\beta^2}{2 \sqrt{m_{\beta} }}
 \int_{\R} \kappa(s) F(s)d s,$
which will give us bound states or anti-bound states depending on the interaction between $\kappa$ and $F$. If the function $F$ is not zero everywhere, we can construct a waveguide such the interaction term is positive (respectively, negative) by choosing $\kappa$ with support when $F$ is positive (respectively, negative).  This implies in particular that bending does not act necessarily as an attractive potential for periodically twisted tubes. 

If $\int\tau-\varepsilon\neq0$, $\kappa\neq0$, we can play with the constant $\beta$ to get bound or anti-bound states. For example if $\beta$  is small the main contribution will come from $\int\tau-\varepsilon$, or even we can have $\mu_1=0$.

\subsection{Theorem \ref{T2}} Unless the previous case, for the non-periodically twisted deformed tube, it is known that bending acts as an attractive potential and twisting as a repulsive potential (see \cite{Kre07} for a review). In consequence, when only bending is considered, we know that there exists at least one discrete eigenvalue below $\cE_1$ \cite{DuEx95,Gru04}. In that case, our formulas \eqref{19dic18}-\eqref{19dic18a} show that we have an eigenvalue of $H_\delta$ of the form
$$\cE_1-\displaystyle{\frac{1}{64}\left(\sum_{q\neq1}(\mathcal{E}_{q}-\mathcal{E}_{1})^2 \langle\psi_{q}|t_2\psi_{1}\rangle^2 \langle{\kappa}| (D_{s}^2+\mathcal{E}_{q}-\mathcal{E}_{1})^{-1}{\kappa}\rangle\right)^2}\,\delta^4+O(\delta^5),$$
which coincides with \cite[Theorem 2.2]{Gru04}. 

With our result we can say more than that, for instance taking   $n=2$,  Theorem \ref{T2} yields  that the resonance near  $\cE_2$ satisfies 
$$k(\delta) = -\displaystyle{\frac{i}{8}\sum_{q\neq2}(\mathcal{E}_{q}-\mathcal{E}_{2})^2 \langle\psi_{q}|t_2\psi_{2}\rangle^2 \langle{\kappa}| (D_{s}^2+\mathcal{E}_{q}-\mathcal{E}_{2})^{-1}{\kappa}\rangle}\,\delta^2+O(\delta^3).$$
The real part of the $\delta^2$ term  can come only from $ i\langle{\kappa}| (D_{s}^2+\mathcal{E}_{1}-\mathcal{E}_{2})^{-1}{\kappa}\rangle$,
for which it is easy to give conditions such that this is not equal to zero. In this situation we will have a resonance  which is not an anti-bound state. A related interesting question is if it is possible to choose $\kappa, \tau$ and $ \varepsilon$ such that instead of a resonance we produce an embedded eigenvalue.

In conclusion, what we have shown in \eqref{19dic18}-\eqref{19dic18a} is how the two deformations quantitatively compete to produce resonances, in the regime of a weak coupling constant. The third term in \eqref{19dic18a} appear as an interaction between the two effects.

\appendix 
\section{Some explicit expansions}\label{ap}
\subsection{The perturbation as a second order differential operator}
\label{ap1}
In this appendix we give some explicit computations that are straightforward. 
For simplicity we set $\xi(s):=\tau(s)-\epsilon(s)-\beta$. 
Also it is easy to see that $h(s,t)=1-\kappa(s)(t_2\cos(\theta(s))+t_3\sin(\theta(s)))$ satisfies
\begin{equation*}
(\partial_\varphi h)(s,t)=\kappa(s)(t_3\cos(\theta(s))-t_2\sin(\theta(s)))=:\tilde{h}(s,t)\ \text{ and } 
\partial_\varphi \tilde{h}=1-h \ .
\end{equation*}
With this we can start computing an expression for $W$. Here we systematically use the notation $\dot{f}=\partial_s f$ and $\tilde{f}=\partial_\varphi f$. By definition we have:
\begin{equation*}
W=-\frac{\kappa^2}{4h^2}-(h^{-\frac12}(\partial_s+\xi\partial_\varphi)h^{-\frac12})^2+\partial_s^2-2\beta\partial_s\partial_\varphi+\beta^2\partial_\varphi^2
\end{equation*}
Noticing that
\begin{equation*}
h^{-\frac12}(\partial_s+\xi\partial_\varphi)h^{-\frac12}=\frac1h\partial_s+\frac\xi h\partial_\varphi-\frac{\dot{h}+\xi\tilde{h}}{2h^2} \ ,
\end{equation*}
we get 
\begin{align*}
&\left(h^{-\frac12}(\partial_s+\xi\partial_\varphi)h^{-\frac12}\right)^2=\\
&\frac1{h^2}\partial_s^2-\frac{\dot{h}}{h^3}\partial_s+\frac\xi {h^2}\partial_s\partial_\varphi+\left(\frac{\dot{\xi}}{h^2}-\frac{\xi\dot{h}}{h^3}\right)\partial_\varphi-\frac{\dot{h}+\xi\tilde{h}}{2h^3}\partial_s-\frac{\ddot{h}+\dot{\xi}\tilde{h}+\xi\dot{\tilde{h}}}{2h^3}+\frac{\dot{h}^2+\xi\tilde{h}\dot{h}}{h^4} \\
&+\frac\xi {h^2}\partial_s\partial_\varphi-\frac{\xi\tilde{h}}{h^3}\partial_s +\frac{\xi^2}{h^2}\partial_\varphi^2-\frac{\xi^2\tilde{h}}{h^3}\partial_\varphi-\frac{\xi\dot{h}+\xi^2\tilde{h}}{2h^3}\partial_\varphi-\frac{\xi\dot{\tilde{h}}+\xi^2(1-h)}{2h^3}+\frac{\xi\dot{h}\tilde{h}+\xi^2\tilde{h}^2}{h^4}\\
&-\frac{\dot{h}+\xi\tilde{h}}{2h^3}\partial_s-\frac{\xi\dot{h}+\xi^2\tilde{h}}{2h^3}\partial_\varphi+\frac{\dot{h}^2}{4h^4}+\frac{\xi\dot{h}\tilde{h}}{2h^4}+\frac{\xi^2\tilde{h}^2}{4h^4}\ .
\end{align*}

Then, writing
\begin{equation}
\label{E:devW}
W=f_{0,0}+f_{1,0}\partial_s+f_{0,1}\partial_\varphi+f_{1,1}\partial_s\partial_\varphi+f_{2,0}\partial_s^2+f_{0,2}\partial_\varphi^2
\end{equation}
we obtain
\begin{align*}
f_{0,0}&=-\frac{\kappa^2}{4h^2}+\frac{\ddot{h}}{2h^3}+\frac{\dot\xi\tilde{h}}{2h^3}+\frac{\xi\dot{\tilde{h}}}{h^3}-\frac{5\dot{h}^2}{4h^4}-\frac{5\xi\tilde{h}\dot{h}}{2h^4}+\frac{\xi^2(1-h)}{2h^3}-\frac{5\xi^2\tilde{h}^2}{4h^4}\ ;\\
f_{1,0}&=\frac{2\dot{h}}{h^3}+\frac{2\xi\tilde{h}}{h^3}\ ;\\
f_{0,1}&=\frac{2\xi\dot{h}}{h^3}+\frac{2\xi^2\tilde{h}}{h^3}-\frac{\dot{\xi}}{h^2}\ ;\\
f_{1,1}&=-2\beta-\frac{2\xi}{h^2}\ ;\\
f_{2,0}&=1-\frac1{h^2}\ ;\\
f_{0,2}&=\beta^2-\frac{\xi^2}{h^2}\ .
\end{align*}

\subsection{Asymptotics of the coefficient in the perturbative regime}\label{ap2}
 We are now interested in  the  asymptotic behavior of $\langle\eta\otimes\psi_n|\eta^{-1}W_\delta\psi_n\rangle=\langle\psi_n|W_\delta\psi_n\rangle$ as $\delta\to 0$. For this we need to make appear the $\delta$-dependence in the previous expressions. We set $\xi_\delta(s)=\delta\tau(s)-\delta\epsilon(s)-\beta$ and introduce the auxiliary function $\Xi_\delta(s):=\xi_\delta(s)+\beta=\delta\tau(s)-\delta\epsilon(s)=:\delta\Xi(s)$. Furthermore, setting $E(s)=\int_{-\infty}^s\epsilon(s)ds$ we get the expression $\theta_\delta(s)=\beta s+\delta E(s)$. Before studying $\langle \psi_n|W_\delta\psi_n\rangle$ we need the asymptotic behavior of $h_\delta(s,t)=1-\delta\kappa(s)(t_2\cos(\theta_\delta(s))+t_3\sin(\theta_\delta(s)))$ given by

\bel{hdelta}
h_\delta=1+\delta g_1+\delta^2 g_2+O(\delta^3)\ee where
\bel{g1g2}\begin{aligned}
g_1(s,t)&=-\kappa(s)(t_2\cos(\beta s)+t_3\sin(\beta s))\ ,\\
g_2(s,t)&=-\kappa(s)E(s)(-t_2\sin(\beta s)+t_3\cos(\beta s))\ .\end{aligned}\ee
We will also need the fact that:
\begin{align*}
h_\delta^{-1}&=1-g_1\delta+(g_1^2-g_2)\delta^2+\cO(\delta^3)\ ;\\
h_\delta^{-2}&=1-2g_1\delta+(3g_1^2-2g_2)\delta^2+\cO(\delta^3)\ ;\\
h_\delta^{-n}&=1-ng_1\delta+O(\delta^2)\ .
\end{align*}
The relation $\partial_\varphi h=\tilde{h}$ gives the corresponding asymptotic for $\tilde{h}_\delta$
\begin{equation*}
\tilde{h}_\delta=\tilde{g}_1\delta+\tilde{g}_2\delta^2+O(\delta^3),
\end{equation*}
where $\tilde{g}_i=\partial_\varphi g_i$. 

 \begin{lemma}\label{le_app}
 Under the assumptions of decay of $\kappa,\tau,\varepsilon $ and its derivatives (see \eqref{hypeps}) we have the following asymptotic:
 \begin{equation}
\langle \psi_n|W_\delta\psi_n\rangle=\breve{\mu}_{1,n}\delta+\breve{\mu}_{2,n}\delta^2+O(\delta^3) .
 \end{equation}
The constant $\breve{\mu}_{1,n}$ is given by
 \begin{equation}
\breve{\mu}_{1,n}=2\beta\|\partial_\varphi\psi_n\|^2 \int_\R \varepsilon-\tau +\beta^2\int_{\R\times\omega} \kappa (\tfrac{1}{2}|\psi_n|^2+2|\partial_\varphi\psi_n|^2)\vartheta,
 \end{equation}
 where  $\vartheta(s,t)$ is given by
 \bel{D:vartheta}
\vartheta(s,t):=(t_2\cos(\beta s)+t_3\sin(\beta s)).
\ee
 
 Assume moreover that $\beta=0$. Then $\breve{\mu}_{1,n}=0$ and $\breve{\mu}_{2,n}$ is given by
 \begin{equation}
\breve{\mu}_{2,n}=-\left\|\frac\kappa2\right\|^2+\left\|\frac{\dot{\kappa}}2t_2\psi_n+(\tau-\varepsilon)\partial_\varphi \psi_n\right\|^2.
\end{equation}
 \end{lemma}
\begin{proof}

In order to compute the differential operator in the r.h.s.\ of \eqref{E:devW} applied to $\psi_n$, since $\partial_s \psi_n=0$, we only need the expansions of $f_{0,j}$ for $j\in\{0,1,2\}$:

\begin{align*}
f_{0,0}(\delta)=&\left(-\frac{\kappa^2}4\delta^2 +O(\delta^3)\right)+\left(\frac{\ddot{g}_1}2\delta+\frac{\ddot{g}_2-3g_1{\ddot{g}_1}}2\delta^2+O(\delta^3) \right)+\left(\frac{\dot{\Xi}\tilde{g}_1}2\delta^2+O(\delta^3) \right)\\
&+\left(-\beta\dot{\tilde{g}}_1\delta+(3\beta g_1\dot{\tilde{g}}_1-\beta\dot{\tilde{g}}_2+\dot{\tilde{g}}_1\Xi)\delta^2+O(\delta^3) \right)+\left(\frac{-5\dot{g}_1^2}4\delta^2+O(\delta^3) \right)\\
&+\left(\frac{5\beta \dot{g}_1\tilde{g}_1}2\delta^2+O(\delta^3) \right)+\left(\frac{-\beta^2g_1}2\delta+\frac{\beta^2(3g_1^2-g_2)+2\beta\Xi g_1}2\delta^2+O(\delta^3) \right)\\
&+\left(\frac{-5\beta^2\tilde{g}_1^2}4\delta^2+O(\delta^3) \right) 
 \end{align*}
 which gives 
 \begin{align*}
f_{0,0}(\delta)=&\left(\frac{\ddot{g}_1}2-\frac{\beta^2g_1}2-\beta\dot{\tilde{g}}_1\right)\delta+\Bigg[\frac12\left(\ddot{g}_2-3g_1{\ddot{g}_1}+\dot{\Xi}\tilde{g}_1+5\beta \dot{g}_1\tilde{g}_1+\beta^2(3g_1^2-g_2)\right)\\
&-\frac54\left( \beta^2\tilde{g}_1^2+\dot{g}_1^2\right)+\beta\Xi g_1+3\beta g_1\dot{\tilde{g}}_1-\beta\dot{\tilde{g}}_2+\Xi\dot{\tilde{g}}_1-\frac{\kappa^2}4\bigg]\delta^2+O(\delta^3).
\end{align*}
The higher order terms are
 \begin{align*}
f_{0,1}(\delta)=&\left(2\beta^2\tilde{g}_1-2\beta\dot{g}_1-\dot{\Xi}\right)\delta
\\&+\left(2\Xi\dot{g}_1+2\dot{\Xi}g_1-4\beta\Xi\tilde{g}_1-2\beta\dot{g}_2+6\beta g_1\dot{g}_1+2\beta^2\tilde{g}_2-6\beta^2g_1\tilde{g}_1)\right)\delta^2+O(\delta^3)
\end{align*}
and
$$
f_{0,2}(\delta)=(2\beta\Xi+2\beta^2g_1)\delta-\left(\Xi^2+4\beta\Xi g_1+\beta^2(3g_1^2-2g_2)\right)\delta^2+O(\delta^3)
$$

Let us now compute  $\breve{\mu}_{1,n}$ when $\beta\neq0$. Notice that for any function 
$\R\times\omega\ni (s,t)\to F(s,t)$ such that $\dot{F}(\cdot,t)$ is integrable for every $t\in\omega$, we have $\langle \dot{F}\psi_n,\psi_n\rangle=0$ for every $1\leq n\leq\infty$. Therefore, we have 
$$\langle \psi_n|f_{0,0}\psi_n\rangle=-\frac{1}{2}\beta^2\delta\int_{\R\times\omega} g_1 \psi_n^2+O(\delta^2) \ \mbox{and} \ \langle \psi_n|f_{0,1}\partial_\varphi\psi_n\rangle=2\beta^2\delta\int_{\R\times\omega} \tilde{g}_1(\partial_\varphi\psi_n) \psi_n+O(\delta^2)\ .$$
An integration by parts provides
\begin{align*}
\langle \psi_n|f_{0,2}\partial^2_\varphi\psi_n\rangle=&\int_{\R\times\omega}2\delta\beta\Xi(\partial^2_\varphi\psi_n)\psi_n+2\delta\beta^2g_1\partial^2_\varphi\psi_n \psi_n+O(\delta^2)\\
=&-\int_{\R\times\omega}2\delta\beta\Xi(\partial_\varphi\psi_n)^2+2\delta\beta^2\left(\tilde{g}_1(\partial_\varphi\psi_n) \psi_n+g_1(\partial_\varphi\psi_n)^2\right)+O(\delta^2).
\end{align*}
Combining these results we get
\begin{equation}\label{mu1}
\langle \psi_n,W_\delta\psi_n\rangle=\left(\frac{-\beta^2}2\int_{\R\times\omega} g_1 |\psi_n|^2-2\int_{\R\times\omega}(\beta\Xi +\beta^2g_1)|\partial_\varphi\psi_n|^2\right)\delta+O(\delta^2).\end{equation}
from where one easily deduces the expression for $\breve{\mu}_{1,n}$. 

Assume now that $\beta=0$. It is clear from \eqref{mu1} that $\mu_{1,n}=0$.
Let us now compute the expression of $\mu_{2,n}$. We have 
\begin{align}
\langle \psi_n |f_{0,0} \psi_n \rangle=&\delta^2\int_{\R\times\omega} \left(-\frac{\kappa^2}{4}-\frac{3}{2}g_1 \ddot{g_1}+\left(\frac12\dot{\Xi}\tilde{g_1}+\Xi\dot{\tilde{g_1}}\right)-\frac{5}{4}\dot{g_1}^2\right) \psi_n ^2+O(\delta^3)\nonumber
\\
=&\delta^2\int_{\R\times\omega} \left(-\frac{\kappa^2}{4}+\frac{1}{4}\dot{g_1}^2+\frac12{\Xi} \dot{\tilde{g_1}}\right) \psi_n ^2+O(\delta^3)\label{beta0delta2a}
\end{align}
where we have used an integration by parts in $s$ and that $\int_{\R} \dot{\Xi}\tilde{g_1}+\dot{\tilde{g_1}}\Xi=0 $. Moreover, 
\bel{beta0delta2b}\langle \psi_n | f_{0,1}\partial_\varphi\psi_n \rangle=\cO(\delta^3) \ \mbox{and} \ \langle f_{0,2}\partial^2_\varphi\psi_n |\psi_n \rangle=\delta^2\int_{\R\times\omega} \Xi^2 (\partial_\varphi \psi_n)^2+O(\delta^3).\ee
 Combining \eqref{beta0delta2a} and \eqref{beta0delta2b} and integrating by parts we deduce $\breve{\mu}_{2,n}$ by
\begin{align*}
\langle \psi_n |W_\delta \psi_n \rangle=&\delta^2\int_{\R\times\omega} -\frac{\kappa^2}{4}\psi_n ^2+\frac{1}{4}\dot{g_1}^2\psi_n ^2-{\Xi} \dot{g_1}(\partial_\varphi\psi_n)\psi_n +\Xi^2 (\partial_\varphi \psi_n)^2+O(\delta^3)\\
=&\delta^2\left(-\left\|\frac\kappa2\right\|^2+\int_{\R\times\omega}\left(\frac{\dot{g}_1\psi_n}2\right)^2-{\Xi} \dot{g_1}(\partial_\varphi\psi_n)\psi_n +(\Xi\partial_\varphi \psi_n)^2\right)+O(\delta^3)\\
=&\delta^2\left(-\left\|\frac\kappa2\right\|^2+\int_{\R\times\omega}\left(\frac{\dot{g}_1\psi_n}2-\Xi\partial_\varphi \psi_n\right)^2\right)+O(\delta^3)\ .\qedhere
\end{align*}
\end{proof}

\bibliographystyle{plain}
\bibliography{bibliopof}

\begin{thebibliography}{10}

\bibitem{BakEx18}
Fedor~L. Bakharev and Pavel Exner.
\newblock Geometrically induced spectral effects in tubes with a mixed
  {D}irichlet-{N}eumann boundary.
\newblock {\em Rep. Math. Phys.}, 81(2):213--231, 2018.

\bibitem{BonBruRai07}
Jean-Fran{\c{c}}ois Bony, Vincent Bruneau, and Georgi Raikov.
\newblock Resonances and spectral shift function near the {L}andau levels.
\newblock {\em Annales de l'institut Fourier}, 57(2):629--672, 2007.

\bibitem{BrHaKr15}
Philippe Briet, Hiba Hammedi, and David Krej{\v{c}}i{\v{r}}{\'\i}k.
\newblock {H}ardy inequalities in globally twisted waveguides.
\newblock {\em Lett. Math. Phys.}, 105:939--958, 2015.

\bibitem{BriKovRaiSoc09}
Philippe Briet, Hynek Kova\v{r}{\'\i}k, Georgi Raikov, and Eric Soccorsi.
\newblock Eigenvalue asymptotics in a twisted waveguide.
\newblock {\em Comm. Partial Differential Equations}, 34(7-9):818--836, 2009.

\bibitem{BriKovRai14}
Philippe Briet, Hynek Kova\v{r}{\'\i}k, and Georgi~D. Raikov.
\newblock Scattering in twisted waveguides.
\newblock {\em J. Funct. Anal.}, 266, 2014.

\bibitem{BruMirPof18}
Vincent Bruneau, Pablo Miranda, and Nicolas Popoff.
\newblock Resonances near thresholds in slightly twisted waveguides.
\newblock {\em Proc. Amer. Math. Soc.}, 146(11):4801--4812, 2018.

\bibitem{CheDuFrei05}
Boris Chenaud, Pierre Duclos, Pedro Freitas, and David
  Krej{\v{c}}i{\v{r}}{\'{\i}}k.
\newblock Geometrically induced discrete spectrum in curved tubes.
\newblock {\em Differential Geom. Appl.}, 23(2):95--105, 2005.

\bibitem{DuEx95}
Pierre Duclos and Pavel Exner.
\newblock Curvature-induced bound states in quantum waveguides in two and three
  dimensions.
\newblock {\em Rev. Math. Phys.}, 7(1):73--102, 1995.

\bibitem{EkKovKre08}
Tomas Ekholm, Hynek Kova{\v{r}}{\'\i}k, and David Krej{\v{c}}i{\v{r}}{\'\i}k.
\newblock A hardy inequality in twisted waveguides.
\newblock {\em Archive for Rational Mechanics and Analysis}, 188(2):245--264,
  2008.

\bibitem{ExKov15}
Pavel Exner and Hynek Kova{\v{r}}{\'\i}k.
\newblock {\em Quantum waveguides}.
\newblock Theoretical and Mathematical Physics. Springer, 2015.

\bibitem{ExKov05}
Pavel Exner and Hynek Kova\v{r}\'{i}k.
\newblock Spectrum of the {S}chr\"{o}dinger operator in a perturbed
  periodically twisted tube.
\newblock {\em Lett. Math. Phys.}, 73(3):183--192, 2005.

\bibitem{Gru04}
Viktor~V. Grushin.
\newblock On the eigenvalues of finitely perturbed laplace operators in
  infinite cylindrical domains.
\newblock {\em Mathematical Notes}, 75(3-4):331--340, 2004.

\bibitem{Gru05}
Viktor~V. Grushin.
\newblock Asymptotic behavior of the eigenvalues of the {S}chr\"odinger
  operator with transversal potential in a weakly curved infinite cylinder.
\newblock {\em Math. Notes}, 77:606--613, 2005.

\bibitem{Kre07}
David Krej{\v{c}}i{\v{r}}{\'{\i}}k.
\newblock Twisting versus bending in quantum waveguides.
\newblock In {\em Analysis on graphs and its applications}, volume~77 of {\em
  Proc. Sympos. Pure Math.}, pages 617--637. Amer. Math. Soc., Providence, RI,
  2008.

\bibitem{KS08}
David Krej\v{c}i\v{r}\'{i}k and Helena \v{S}ediv\'{a}kov\'{a}.
\newblock The effective {H}amiltonian in curved quantum waveguides under mild
  regularity assumptions.
\newblock {\em Rev. Math. Phys.}, 24(7):1250018, 39, 2012.

\bibitem{RS3}
Michael Reed and Barry Simon.
\newblock {M}ethods of modern mathematical physics, vol. {III}, {S}cattering
  theory.
\newblock {\em Bull. Amer. Math. Soc. 2 (1980)}, pages 0273--0979, 1980.

\end{thebibliography}
\end{document}